\documentclass{article}
\usepackage{amsfonts}
\usepackage{amsmath}
\usepackage{amssymb}
\usepackage{amsthm}
\usepackage{graphicx}
\RequirePackage{doi}
\usepackage{hyperref}
\usepackage{xcolor}
\usepackage{comment}
\usepackage[style=alphabetic,backend=biber]{biblatex}
\newtheorem{theorem}{Theorem}

\newtheorem{prop}{Theorem}
\newtheorem{proposition}[prop]{Proposition}

\newtheorem{cor}{Theorem}
\newtheorem{corollary}[cor]{Corollary}

\newtheorem{thm}{Theorem}
\newtheorem{definition}[thm]{Definition}

\newtheorem{ex}{Theorem}
\newtheorem{example}[ex]{Example}

\newtheorem{lem}{Theorem}
\newtheorem{lemma}[lem]{Lemma}

\newtheorem{rem}{Theorem}
\newtheorem{remark}[rem]{Remark}

\providecommand{\keywords}[1]
{
  \small	
  {\textit{Keywords:}} #1
}

\providecommand{\msc}[1]
{
  \small	
  {\textit{MSC2020:}} #1
}

\title{Symmetric Bessmertny{\u\i} Realizations and Field Extension Problems in Characteristic 2 - A Differential Algebra Approach}

\author{Soumya Sinha Babu$\footnote{Email: ssinhababu@fit.edu}$\; and Aaron Welters$\footnote{Email: awelters@fit.edu}$\\ \\Department of Mathematics and Systems Engineering\\Florida Institute of Technology\\Melbourne, FL 32901, USA}
\date{}

\begin{document}

\maketitle

\begin{abstract}
We present a short, purely algebraic proof of the Symmetric Bessmertny{\u\i} Realization Theorem in the characteristic $2$ case recently proved in \cite{Elsinger:BRS:2026}. Symmetric Bessmertny{\u\i} realizations are Schur complements of affine linear symmetric matrix pencils, and they arise naturally as state-space representations in linear systems theory. In contrast with the algorithmic approach in \cite{Elsinger:BRS:2026}, we use differential algebra: by defining formal partial derivatives on multivariate rational functions over fields of positive characteristic and considering their corresponding field of constants, we obtain scalar criteria for symmetric and homogeneous symmetric realizability in characteristic $2$, effectively reducing the matrix-valued problem to its diagonal entries. As a consequence, we prove a new theorem on the field extension problem for symmetric and homogeneous symmetric Bessmertny{\u\i} realizations. Finally, in the scalar case, we identify realizable rational functions with vector spaces over appropriate fields of constants and quantify the abundance of counterexamples in characteristic $2$.
\end{abstract}


\msc{Primary 15A54; Secondary 15A22, 15B57, 12H05, 13N15, 16W10, 93B25, 93C35, 93B15, 12F20}

\section{Introduction}
Let $\mathbb{F}$ be a field with characteristic denoted by $\operatorname{char}(\mathbb{F})$ and set $z=(z_1,\ldots,z_n)$, where $z_1,\ldots,z_n$ are $n$ indeterminates. Then $\mathbb{F}[z]$ will denote the integral domain of all polynomials in $z_1,\ldots,z_n$ with coefficients in $\mathbb{F}$. Next, we will denote the field of rational functions as $\mathbb{F}(z)$ and the subset of all homogeneous rational functions of degree $d$ as $\mathbb{F}(z)_d$.

The set of all $m\times k$ matrices with entries in $\mathbb{F}$ will be denoted by $\mathbb{F}^{m\times k}$, and if $A=[a_{ij}]\in \mathbb{F}^{m\times k}$ is a matrix then $A^T=[a_{ji}]\in \mathbb{F}^{k\times m}$ will denote its transpose \{and, similarly for $\mathbb{F}[z]$, $\mathbb{F}(z)$, or $\mathbb{F}(z)_d$\}. We now introduce the following important definitions.

\begin{definition}\label{def:LinearPencils}
An (affine) linear matrix pencil (LP) is an element $A(z)\in\mathbb{F}[z]^{m\times m}$ of the form \[A(z) = A_0 + z_1A_1 + \cdots + z_nA_n,\] where $A_0, A_1, \ldots, A_n\in \mathbb{F}^{m\times m}$. If $A_j$ is a symmetric matrix for each $j=0,1,\ldots, n$, then $A(z)$ is called a linear symmetric matrix pencil (sLP). If $A_0=0$, then $A(z)$ is called a linear homogeneous matrix pencil (hLP) and if, in addition, it is a sLP then it is called a linear homogeneous symmetric matrix pencil (hsLP).
\end{definition}

\begin{definition}\label{def:BRDefsAndDefsForFieldExtIssue} 
An element $F(z)\in\mathbb{F}(z)^{k\times k}$ is \textbf{Bessmertny\u{\i} realizable}, i.e., has a \textbf{Bessmertny\u{\i} realization {\upshape{(}}BR{\upshape{)}}} [over the field $\mathbb{F}(z)$], if there exists a linear matrix pencil $A(z)=[A_{ij}(z)]_{i,j=1,2}\in \mathbb{F}[z]^{m\times m}$ partitioned into a $2\times 2$ block matrix form such that $F(z)$ is the Schur complement of $A(z)$ with respect to $A_{22}(z)$, that is, 
\begin{align*}
    F(z)=A(z)/A_{22}(z)=A_{11}(z)-A_{12}(z)[A_{22}(z)]^{-1}A_{21}(z),
\end{align*}
in which case we call $A(z)$ the \textbf{Bessmertny\u{\i} realizer} of $F(z)$.
If $A(z)$ is a sLP then $F(z)$ is said to have a \textbf{symmetric Bessmertny\u{\i} realization {\upshape{(}}SBR{\upshape{)}}}. If $A(z)$ is a hLP then $F(z)$ is said to have a \textbf{homogeneous Bessmertny\u{\i} realization {\upshape{(}}hBR{\upshape{)}}}. If $A(z)$ is a hsLP then $F(z)$ is said to have a \textbf{homogeneous symmetric Bessmertny\u{\i} realization {\upshape{(}}hSBR{\upshape{)}}}. 
\end{definition}

In \cite{Elsinger:BRS:2026}, the following theorem was proved; we will call it the Bessmertny\u{\i}-Elsinger-Orzel-Welters Theorem, or the \textit{Bessmertny\u{\i} Realization Theorem} for short.
\begin{theorem}[Elsinger-Orzel–Welters (2026)]\label{thm:BessmertnyiRealizabilityThm}
Let $F(z)\in\mathbb{F}(z)^{k\times k}$. Then:
\begin{itemize}
\item[$($a$)$] $F(z)$ has a BR.
\item[$($b$)$] $F(z)$ has a hBR if and only if $F(z)\in \mathbb{F}(z)_1^{k\times k}$. 
\item[$($c$)$] $F(z)$ has a SBR if and only if $F(z)^T=F(z)$ and any one of the following three conditions holds: {\upshape{(}}i{\upshape{)}} $n=1;$ {\upshape{(}}ii{\upshape{)}} $\operatorname{char}(\mathbb{F})\not=2;$ {\upshape{(}}iii{\upshape{)}} $n\geq 2$, $\operatorname{char}(\mathbb{F})=2,$ and the diagonal entries of $F(z)$ are in the subspace
\begin{gather}
    \operatorname{span}\{q_0^2+z_1q_1^2+\cdots+z_nq_n^2:q_i\in \mathbb{F}(z)\},\label{SBRSubspaceChar2}
\end{gather}
where the span is over the field $\mathbb{F}$.
\item[$($d$)$] $F(z)$ has a hSBR if and only if $F(z)\in\mathbb{F}(z)_1^{k\times k}$, $F(z)^T=F(z)$, and any one of the following three conditions holds: {\upshape{(}}i{\upshape{)}} $n=1,2;$ {\upshape{(}}ii{\upshape{)}} $\operatorname{char}(\mathbb{F})\not=2;$ {\upshape{(}}iii{\upshape{)}} $n\geq 3$, $\operatorname{char}(\mathbb{F})=2,$ and the diagonal entries of $F(z)$ are in the subspace
\begin{gather}
    \operatorname{span}\{z_1q_1^2+\cdots+z_nq_n^2:q_i\in \mathbb{F}(z)_0\},\label{hSBRSubspaceChar2}
\end{gather}
where the span is over the field $\mathbb{F}$.
\end{itemize}
\end{theorem}

The original proof of this theorem in \cite{Elsinger:BRS:2026} was algorithmic and rather long (about 26 pages) with most of the effort and pages dedicated to proving statement (c)(iii) in the case $k=1$. Thus, the main goal of our paper is to give an alternative, purely algebraic proof of this important theorem that simplifies and shortens the one from \cite{Elsinger:BRS:2026} and to also provide more clarity and insight on why the statement is true. As we shall see we can achieve this by proving the following theorem using a new approach based on differential algebra.
\begin{theorem}\label{thm:MainResultSBRChar2ScalarCase}
    Let $F(z)\in\mathbb{F}(z)^{1\times 1}$ and suppose $\operatorname{char}(\mathbb{F})=2$. Then the following statements are true: 
    \begin{itemize}
        \item[(i)] If $F(z)$ has a SBR then
\begin{gather*}
    F(z)\in\operatorname{span}\{q_0^2+z_1q_1^2+\cdots+z_nq_n^2:q_i\in \mathbb{F}(z)\},
\end{gather*}
where the span is over the field $\mathbb{F}$. 
        \item[(ii)] If $F(z)$ has a hSBR then
\begin{gather*}
    F(z)\in\operatorname{span}\{z_1q_1^2+\cdots+z_nq_n^2:q_i\in \mathbb{F}(z)_0\},
\end{gather*}
where the span is over the field $\mathbb{F}$.  
    \end{itemize}
\end{theorem}

One way to see how differential algebra is involved when $\operatorname{char}(\mathbb{F})=2$ is to note the following two lemmas that we use to prove Theorem \ref{thm:MainResultSBRChar2ScalarCase}. The first lemma is well-known, but for the reader's convenience we will prove it in Sec.\ \ref{sec:DifferentialAlgebra} (see Proposition \ref{prop:FieldOFConstantsForPartialDerivatives}). The second lemma we will prove in Sec.\ \ref{sec:PrfMainThmForSBRUsingDiffAlg} before giving the proof of Theorem \ref{thm:MainResultSBRChar2ScalarCase}.
\begin{lemma}\label{lem:FormalParitalDervFieldOfConstants}
    If $\operatorname{char}(\mathbb{F})=p$ and $\partial/\partial z_1,\ldots, \partial/\partial z_n$ are the formal partial derivatives for the field $\mathbb{F}(z)$ then
    \begin{gather}
        \mathbb{F}(z^p)=\left\{r(z)\in \mathbb{F}(z):\frac{\partial}{\partial z_i}r(z) =0, \text{ for each }i=1,\ldots, n\right\},
    \end{gather}
    where $\mathbb{F}(z^p)$ is the rational field in $z^p=(z_1^p,\ldots, z_n^p)$.
\end{lemma}
\begin{lemma}\label{lem:PartialDervZeroConnection2SBR}
    If $\operatorname{char}(\mathbb{F})=2$ then \begin{gather}
    \operatorname{span}\{q^2:q\in \mathbb{F}(z)\}=\mathbb{F}(z^2),\; \operatorname{span}\{q^2:q\in \mathbb{F}(z)_0\}=\mathbb{F}(z^2)_0,\label{AltReprSpanSquare}\\
    \operatorname{span}\{q_0^2+z_1q_1^2+\cdots+z_nq_n^2:q_i\in \mathbb{F}(z)\}=\mathbb{F}(z^2)+z_1\mathbb{F}(z^2)+\cdots+z_n\mathbb{F}(z^2),\label{AltReprSpanOfSquares}\\
    \operatorname{span}\{z_1q_1^2+\cdots+z_nq_n^2:q_i\in \mathbb{F}(z)_0\}=z_1\mathbb{F}(z^2)_0+\cdots+z_n\mathbb{F}(z^2)_0,\label{AltReprhSBRSpanOfSquares}
\end{gather}
where $\mathbb{F}(z^2)$ is the rational field in $z^2=(z_1^2,\ldots, z_n^2)$ with $\mathbb{F}(z^2)_0:=\mathbb{F}(z^2)\cap \mathbb{F}(z)_0$ the subfield of $\mathbb{F}(z^2)$ of homogeneous degree-zero elements.
\end{lemma}

There are two unexpected applications of this approach. The first is that we know when $\operatorname{char}(\mathbb{F})=2$ there are many more non-SBRs than SBRs and we can quantify this as follows in the case $k=1$. First, $\mathbb{F}(z)$ is a $2^n$-dimensional vector space over $\mathbb{F}(z^2)$ (this is well-known, for instance, see \cite[Prop.\ 2.2]{Santos:2023:PCD}, but for the benefit of the reader we prove it in Sec.\ \ref{sec:AppendixAuxPrfs}) and \eqref{AltReprSpanOfSquares} is an $(n+1)$-dimensional subspace of $\mathbb{F}(z)$ over $\mathbb{F}(z^2)$ (by Corollary \ref{cor:2ndMainCor}). From which it follows by Lemma \ref{lem:PartialDervZeroConnection2SBR} and Theorem \ref{thm:BessmertnyiRealizabilityThm} that there are many more non-SBR elements (within the space of all rational functions over the field $\mathbb{F}$) since the ``density" of non-SBR elements is $1$ as $\lim_{n\rightarrow \infty}\frac{n+1}{2^n}=0$. Similarly, it can be shown that $\mathbb{F}(z)_1$ is a $2^{n-1}$-dimensional vector space over $\mathbb{F}(z^2)_0$ (as proven in Sec.\ \ref{sec:AppendixAuxPrfs}) and \eqref{AltReprhSBRSpanOfSquares} is an $n$-dimensional subspace of $\mathbb{F}(z)_1$ over $\mathbb{F}(z^2)_0$ (by Corollary \ref{cor:2ndMainCor}) from which it follows by Lemma \ref{lem:PartialDervZeroConnection2SBR} and Theorem \ref{thm:BessmertnyiRealizabilityThm} that there are many more non-hSBR elements (within the space of all homogeneous degree-one rational functions over the field $\mathbb{F}$) since the ``density" of non-hSBR elements is $1$ as $\lim_{n\rightarrow \infty}\frac{n}{2^{n-1}}=0$.

Second, we will also be able to prove the theorem below. It is a new result on the field extension problem for Bessmertny{\u\i} realizations for square matrices with entries in $\mathbb{F}(z_1,\ldots, z_n)$ which arises, for instance, if we are also considering these matrices as having entries in a field extension $\mathbb{K}(z_1,\ldots, z_N)$ (with $N\geq n$) of $\mathbb{F}(z_1,\ldots, z_n)$, where $\mathbb{K}$ is an algebraic field extension of $\mathbb{F}$ (see the end of Sec.\ \ref{sec:AppendixAuxPrfs} for an expanded discussion on this) and $z_1,\ldots, z_N$ are indeterminates. Note that because of Theorem \ref{thm:BessmertnyiRealizabilityThm}, this problem is non-trivial only when $\text{char}(\mathbb{F})=2$.

\begin{theorem}\label{thm:MainResultSBRFieldExtension}
     Let $\mathbb{F}$ be a field, $\mathbb{K}$ a field extension of $\mathbb{F}$, and $z_1,\ldots, z_n,\ldots, z_N$ be indeterminates with $N\geq n$. Set $z=(z_1,\ldots, z_n)$ and $w=(w_1,\ldots, w_N)$ with $w_i=z_i$ for $i=1,\ldots, N$. 
     If  $F(z)\in\mathbb{F}(z)^{k\times k}$ and $F(z)$ has a BR, SBR, hBR, or hSBR over the field $\mathbb{K}(w)$ (cf.\ Def.\ \ref{def:BRDefsAndDefsForFieldExtIssue}) then it has a BR, SBR, hBR, or hSBR, respectively, over the field $\mathbb{F}(z)$.
\end{theorem}

The remainder of the paper is organized as follows. In Sec.\ \ref{sec:DifferentialAlgebra} we give some preliminary results on differential algebras. These are then used in Sec.\ \ref{sec:PrfMainThmForSBRUsingDiffAlg} to prove Theorem \ref{thm:MainResultSBRChar2ScalarCase}. In Sec.\ \ref{sec:SchurComplAlgebra} we recall several basic algebraic results on Schur complements from \cite{Elsinger:BRS:2026} which yields a streamlined proof of Theorem \ref{thm:BessmertnyiRealizabilityThm} except for statements (c)(iii) and (d)(iii), whose proofs are completed in Sec.\ \ref{sec:CompletingPrfMainThm}. Next, in Sec.\ \ref{sec:PrfThmSBRFieldExt}, we prove Theorem \ref{thm:MainResultSBRFieldExtension}. Finally, the paper concludes with an appendix containing auxiliary material that supplements the main discussion.

\section{Differential Algebra}\label{sec:DifferentialAlgebra}
Let us give some preliminaries on differential algebra from the classical books \cite{Kaplansky:1957:IDA, Kolchin:1973:DAA}. For the basic material on abstract algebra such as definitions (and notations) of rings, integral domains, quotient/rational fields, matrix rings, etc., we will use the book \cite{Dummit:2004:AAL}. 
\begin{definition}
    A derivation $d$ of a ring $R$ is a function $d:R\rightarrow R$ satisfying for all $a,b\in R$: 
    \begin{itemize}
        \item[(i)] $d(a+b)=d(a)+d(b)$;
        \item[(ii)] $d(ab)=d(a)b+ad(b)$.
    \end{itemize}
    The set of all derivations of $R$ is denoted by $\operatorname{Der}(R)$. We call $R$ a partial differential ring with partial derivations $\Delta$, provided $\Delta$ is a nonempty finite subset of $\operatorname{Der}(R)$ with the property that $\partial(d(a))=d(\partial(a))$ for all pairs $d,\partial\in \Delta$ and all $a\in R$ (i.e., $\Delta$ is a commuting subset of derivations of $R$); if $R$ is an integral domain or field then it is called a partial differential integral domain or field, respectively. 
\end{definition}

\begin{example}\label{ex:PartialDerivativesAreDerivations}
    Let $\mathbb{F}$ be a field and $z_1,\ldots, z_n$ be indeterminates. Let $\mathbb{F}[z]$ denote the set of all polynomials in $z_1,\ldots, z_n$ [where $z=(z_1,\ldots, z_n)$] with coefficients in $\mathbb{F}$ and denote its rational field by $\mathbb{F}(z)$. Then $\mathbb{F}[z]$ is a partial differential integral domain with a set of formal partial derivations $\Delta=\{\frac{\partial}{\partial z_i}:i=1,\ldots, n\}$ of $\mathbb{F}[z]$ that are uniquely determined by the properties that for all $i,j=1,\ldots, n$: (1) $\frac{\partial}{\partial z_i}(a)=0,$ for all $a\in \mathbb{F}$ and (2) $\frac{\partial}{\partial z_i}(z_j)=\delta_{ij}$ (where $\delta_{ij}$ denotes the Kronecker delta). Next, these derivations can be uniquely extended to partial derivations $\Delta=\{\frac{\partial}{\partial z_i}:i=1,\ldots, n\}$ on $\mathbb{F}(z)$ with the property that for all $i=1,\ldots, n$: (3) $\frac{\partial}{\partial z_i}\left(\frac{p(z)}{q(z)}\right)=\frac{\frac{\partial}{\partial z_i}(p(z))q(z)-p(z)\frac{\partial}{\partial z_i}(q(z))}{q(z)^2},$ for all $p(z),q(z)\in \mathbb{F}[z]$ with $q(z)\not=0$.
    This makes $\mathbb{F}(z)$ a partial differential field with partial derivations $\Delta$. Finally, these derivations can be extended to derivations $\Delta=\{\frac{\partial}{\partial z_i}:i=1,\ldots, n\}$ on the matrix ring $\mathbb{F}(z)^{k\times k}$ (and similarly for $\mathbb{F}[z]^{k\times k}$) by defining for each $l=1,\ldots, n$: $\frac{\partial}{\partial z_l}[r_{ij}(z)]=[\frac{\partial}{\partial z_l}r_{ij}(z)]$ for any matrix $[r_{ij}(z)]\in \mathbb{F}(z)^{k\times k}$. 
\end{example}
From here on out we will just drop the term ``partial" for convenience.

\begin{remark}\label{rem:UsualPartialDerivativesAreLinearOperators}
    Consider the derivations $\Delta=\{\frac{\partial}{\partial z_i}:i=1,\ldots, n\}$ from Example \ref{ex:PartialDerivativesAreDerivations}. In the case that $\mathbb{F}\in \{\mathbb{R},\mathbb{C}\}$, the derivations $\Delta$ of $\mathbb{F}[z], \mathbb{F}(z),$ or $\mathbb{F}(z)^{k\times k}$ are just the usual partial derivatives for such functions $f(z)$, that is, $\partial/\partial z_i f(z)=\lim_{h\rightarrow 0}\frac{f(z+h)-f(z)}{h},$ for each $i=1,\ldots, n$ and hence are linear. In general, it can be shown that for any field $\mathbb{F}$, the derivation $\partial/\partial z_i$ is a linear operator on the vector space $\mathbb{F}[z], \mathbb{F}(z),$ $\mathbb{F}[z]^{k\times k}$, and $\mathbb{F}(z)^{k\times k}$, respectively, over the field $\mathbb{F}$, for each $i=1,\ldots, n$.
\end{remark}

The following gives some well-known properties of differential rings.
\begin{lemma}\label{lem:BasicPropertiesOfDerivations}
    Let $R$ be a differential ring with derivations $\Delta$. Then the following are true:
    \begin{itemize}
        \item[(i)]  $d(0)=0$ for all $d\in \Delta$ and, if $R$ has a multiplicative identity $1$ (i.e., unity element) then $d(1)=0$ for all $d\in \Delta$.
        \item[(ii)] If $a\in R$ has multiplicative inverse $a^{-1}$ then $d(a^{-1})=-a^{-1}d(a)a^{-1}$.
        \item[(iii)] If $R$ is a commutative ring then $d(a^m)=ma^{m-1}d(a)$ for every $m\in\mathbb{N}$ and all $d\in \Delta$.
    \end{itemize}
\end{lemma}
\begin{definition}
    Let $R$ be a differential ring with derivations $\Delta$. An element $c\in R$ is said to be a constant if $d(c)=0, \forall d\in \Delta$.
\end{definition}

The following result is well-known.
\begin{lemma}\label{lem:BasicResultsOnRingOfConstantsOfADifferentialRing}
    Let $R$ be a differential ring with derivations $\Delta$ with set of constants $C=\{c\in R:d(c)=0,\forall d\in \Delta\}=\bigcap_{d\in \Delta}d^{-1}(\{0\})$. Then the following are true:
    \begin{itemize}
        \item[(i)] $C$ is a subring of $R$ and if, in addition, $R$ is an integral domain then $C$ is also.
        \item[(ii)] If $R$ is a field then $C$ is a subfield of $R$.
        \item[(iii)] If $R$ is a commutative ring with characteristic $p$ then $a^p\in C$ for all $a\in R$.
        \item[(iv)] If $f,g\in R$ then $dg=df$ for all $d\in \Delta$ if and only if $g=f+c$ for some $c\in C$.
    \end{itemize}
\end{lemma}

The following result, which has Lemma \ref{lem:FormalParitalDervFieldOfConstants} as a special case, is well-known \cite[Chap.\ 9, Sec.\ 6, Exercises 9 \& 10e, p.\ 524]{Cox:2015:IVA}, \cite[Theorem 2.3]{Okuda:2004:KOD} (see also \cite{Nowicki:1988:ROC, Nowicki:1994:RFC} and references therein), but will be proved here for the reader's convenience.
\begin{proposition}\label{prop:FieldOFConstantsForPartialDerivatives}
    Consider the derivations $\Delta=\{\frac{\partial}{\partial z_i}:i=1,\ldots, n\}$ in Example \ref{ex:PartialDerivativesAreDerivations}, where $\mathbb{F}$ is a field with $\operatorname{char}(\mathbb{F})=p$. Then the following statements are true:
    \begin{itemize}
        \item[(i)] For each $i=1,\ldots, n$, the set of constants of the differential integral domain $\mathbb{F}[z]$ with derivation $\frac{\partial}{\partial z_i}$ is the subring of polynomials  $\mathbb{F}[\zeta_{i1},\ldots, \zeta_{in}],$ where $\zeta_{ij}=z_j$ if $i\not=j$ and $\zeta_{ii}=z_i^p$. Moreover, its quotient field is $\mathbb{F}(\zeta_{i1},\ldots, \zeta_{in})$ which is also the subfield of constants of the differential field $\mathbb{F}(z)$ with derivation $\frac{\partial}{\partial z_i}$.
        \item[(ii)]  The set of constants of the differential integral domain $\mathbb{F}[z]$ with $\Delta=\{\frac{\partial}{\partial z_i}:i=1,\ldots, n\}$ is the subring of all polynomials  $\mathbb{F}[z^p],$ where $z^p=(z_1^p,\ldots, z_n^p)$. Moreover, its quotient field is $\mathbb{F}(z^p)$ which is also the subfield of constants of the differential field $\mathbb{F}(z)$ with derivations $\Delta=\{\frac{\partial}{\partial z_i}:i=1,\ldots, n\}$.
    \end{itemize}
\end{proposition}

\begin{proof}
    (i): First, $\mathbb{F}[z]$ is a vector space over the field $\mathbb{F}$ in which the set of all monomials in $z$, i.e., $\{z^{\alpha}:\alpha \in (\mathbb{N}\cup \{0\})^n\}$ (cf.\ Sec.\ \ref{sec:AppendixAuxPrfs} for this notation), is a basis. Next, for any $\alpha=(\alpha_1,\ldots, \alpha_n) \in (\mathbb{N}\cup \{0\})^n$, we have $\frac{\partial}{\partial z_i}z^{\alpha}=0$ iff $p|\alpha_i$, i.e., $z^{\alpha}\in \mathbb{F}[\zeta_{i1},\ldots, \zeta_{in}]$. It is also clear that $\{\frac{\partial}{\partial z_i}z^{\alpha}=\alpha_iz_i^{\alpha_i-1}\prod_{j=1, j\not=i}^nz_j^{\alpha_j}:\alpha \in (\mathbb{N}\cup \{0\})^n, p\nmid\alpha_i\}$ is a linearly independent set of vectors in $\mathbb{F}[z]$. It follows from these facts and since the derivation $\frac{\partial}{\partial z_i}$ is a linear operator on $\mathbb{F}[z]$ over the field $\mathbb{F}$ (see Remark \ref{rem:UsualPartialDerivativesAreLinearOperators}), that $\mathbb{F}[\zeta_{i1},\ldots, \zeta_{in}]$ is the set of constants of the differential integral domain $\mathbb{F}[z]$ with derivation $\frac{\partial}{\partial z_i}$. By Lemma \ref{lem:BasicResultsOnRingOfConstantsOfADifferentialRing}(i), it follows that $\mathbb{F}[\zeta_{i1},\ldots, \zeta_{in}]$ is a subring of $\mathbb{F}[z]$ which is an integral domain whose quotient field $\mathbb{F}(\zeta_{i1},\ldots, \zeta_{in})$ is a subfield of $\mathbb{F}(z)$. It follows from this and Lemma \ref{lem:BasicResultsOnRingOfConstantsOfADifferentialRing}(ii) that this quotient field is a subset of the field of constants of the differential field $\mathbb{F}(z)$ with derivation $\frac{\partial}{\partial z_i}$. Next, suppose that $r(z)\in \mathbb{F}(z)$ with $r(z)\not=0$ and $\frac{\partial}{\partial z_i}r(z)=0$. Then $r(z)=\frac{a(z)}{b(z)}$ for some $a(z),b(z)\in \mathbb{F}[z]$ with $b(z)\not=0$ and $\gcd(a(z),b(z))=1$. Hence, $0=\frac{\partial}{\partial z_i}r(z)=\frac{b(z)\frac{\partial}{\partial z_i}a(z)-a(z)\frac{\partial}{\partial z_i}b(z)}{b(z)^2}$ implying $b(z)\frac{\partial}{\partial z_i}a(z)=a(z)\frac{\partial}{\partial z_i}b(z)$. As $\gcd(a(z),b(z))=1$, this implies that $a(z)|\frac{\partial}{\partial z_i}a(z)$ and $b(z)|\frac{\partial}{\partial z_i}b(z)$. But if $\frac{\partial}{\partial z_i}a(z)\not=0$ then $\deg_{z_i}\left[\frac{\partial}{\partial z_i}a(z)\right]<\deg_{z_i}\left[a(z)\right]$, a contradiction that $a(z)|\frac{\partial}{\partial z_i}a(z)$. This proves that $\frac{\partial}{\partial z_i}a(z)=0$ and, similarly, $\frac{\partial}{\partial z_i}b(z)=0$. Hence, $a(z),b(z)\in \mathbb{F}[\zeta_{i1},\ldots, \zeta_{in}]$ implying $r(z)=\frac{a(z)}{b(z)}\in \mathbb{F}(\zeta_{i1},\ldots, \zeta_{in})$. Therefore, we have proven that $\mathbb{F}(\zeta_{i1},\ldots, \zeta_{in})$ is the set of constants of the differential field $\mathbb{F}(z)$ with derivation $\frac{\partial}{\partial z_i}$, for each $i=1,\ldots, n$.

    (ii): Consider the differential integral domain $\mathbb{F}[z]$ with derivations $\Delta=\{\frac{\partial}{\partial z_i}:i=1,\ldots, n\}$. Then from statement (i), it follows that the set of its constants is $\bigcap_{d\in \Delta}d^{-1}(\{0\})=\bigcap_{i=1}^n\mathbb{F}[\zeta_{i1},\ldots, \zeta_{in}]=\mathbb{F}[z_1^p,\ldots, z_n^p]=\mathbb{F}[z^p]$. Similarly, for the differential field $\mathbb{F}(z)$ with derivations $\Delta=\{\frac{\partial}{\partial z_i}:i=1,\ldots, n\}$, it follows from statement (i) that its set of constants is $\bigcap_{d\in \Delta}d^{-1}(\{0\})=\bigcap_{i=1}^n\mathbb{F}(\zeta_{i1},\ldots, \zeta_{in})=\mathbb{F}(z_1^p,\ldots, z_n^p)=\mathbb{F}(z^p)$ which is the quotient field of $\mathbb{F}[z^p]$. This completes the proof of the lemma.
\end{proof}

The following result is an immediate corollary of Prop.\ \ref{prop:FieldOFConstantsForPartialDerivatives} and will be needed to prove Theorem \ref{thm:MainResultSBRFieldExtension} in the case $\operatorname{char}(\mathbb{F})=2$.

\begin{corollary}\label{cor:MainCor}
    Let $\mathbb{F}$ be a field with $\operatorname{char}(\mathbb{F})=p$, $\mathbb{K}$ be a field extension of $\mathbb{F}$, and $z_1,\ldots, z_n,\ldots, z_N$ be indeterminates with $N\geq n$. Set $z=(z_1,\ldots, z_n)$ and $w=(w_1,\ldots, w_N)$ with $w_i=z_i$ for $i=1,\ldots, N$. Then the following statements are true:
    \begin{itemize}
        \item[(i)]  If $r(z)\in \mathbb{F}(z)$ and $r(z)\in \mathbb{K}(w^p)$ [$r(z)\in \mathbb{K}(w^p)\cap \mathbb{K}(w)_0$] then $r(z)\in \mathbb{F}(z^p)$ [resp., $r(z)\in \mathbb{F}(z^p)\cap \mathbb{F}(z)_0$].
        \item[(ii)] If $r(z)\in \mathbb{F}(z)$ and $r(z)\in \mathbb{K}(w^p)+w_1\mathbb{K}(w^p)+\cdots+w_N\mathbb{K}(w^p)$ [ $r(z)\in w_1\mathbb{K}(w^p)\cap \mathbb{K}(w)_0+\cdots+w_N\mathbb{K}(w^p)\cap \mathbb{K}(w)_0$] then $r(z)\in \mathbb{F}(z^p)+z_1\mathbb{F}(z^p)+\cdots+z_n\mathbb{F}(z^p)$ [resp., $r(z)\in z_1\mathbb{F}(z^p)\cap \mathbb{F}(z)_0+\cdots+z_n\mathbb{F}(z^p)\cap \mathbb{F}(z)_0$].
    \end{itemize}
\end{corollary}
\begin{proof}
    (i): Suppose $r(z)\in \mathbb{F}(z)$ and $r(z)\in \mathbb{K}(w^p)$ [$r(z)\in \mathbb{K}(w^p)\cap \mathbb{K}(w)_0$]. Then defining the formal partial derivatives $\partial/\partial w_i, i=1,\ldots, N$ as in Example \ref{ex:PartialDerivativesAreDerivations} for $\mathbb{K}(w)$ with $\mathbb{F}$ replaced by $\mathbb{K}$ and $z$ replaced by $w$, it follows by this and the hypotheses that on $\mathbb{F}(z)$ we have $\partial/\partial w_i|_{\mathbb{F}(z)}=\partial/\partial z_i, i=1,\ldots, n$ and $\partial/\partial w_i|_{\mathbb{F}(z)}=0, i>n$; moreover,
    it follows from Prop.\ \ref{prop:FieldOFConstantsForPartialDerivatives}, with $\mathbb{F}$ replaced by $\mathbb{K}$ and $z$ replaced by $w$, that $\mathbb{K}(w^p)$ is the field of constants for the derivations $\{\partial/\partial w_i:i=1,\ldots,N\}$ of $\mathbb{K}(w)$. Hence, we have that $\frac{\partial}{\partial z_i}r(z)=\frac{\partial}{\partial w_i}r(z)=0$ for $i=1,\ldots, n$. From this, it follows by Prop.\ \ref{prop:FieldOFConstantsForPartialDerivatives} [this time for $\mathbb{F}(z)$] that $r(z)\in \mathbb{F}(z^p)$ [and hence if $r(z)\in \mathbb{K}(w^p)\cap \mathbb{K}(w)_0$ then it follows that $r(z)\in \mathbb{F}(z^p)\cap \mathbb{F}(z)_0$].

    (ii): Suppose $r(z)\in \mathbb{F}(z)$ and $r(z)\in \mathbb{K}(w^p)+w_1\mathbb{K}(w^p)+\cdots+w_N\mathbb{K}(w^p)$ [$r(z)\in w_1\mathbb{K}(w^p)\cap \mathbb{K}(w)_0+\cdots+w_N\mathbb{K}(w^p)\cap \mathbb{K}(w)_0$]. Then, for each $i=1,\ldots, n$, we have $\frac{\partial}{\partial z_i}r(z)\in \mathbb{F}(z)$ and also by our hypotheses together with Prop.\ \ref{prop:FieldOFConstantsForPartialDerivatives}, it follows that $\frac{\partial}{\partial z_i}r(z)\in \mathbb{K}(w^p)$ [resp.\ $\frac{\partial}{\partial z_i}r(z)\in \mathbb{K}(w^p)\cap \mathbb{K}(w)_0$] implying by our statement (i) which was just proven, that $\frac{\partial}{\partial z_i}r(z)\in \mathbb{F}(z^p)$ [resp., $\frac{\partial}{\partial z_i}r(z)\in \mathbb{F}(z^p)\cap \mathbb{F}(z)_0$]. Hence, $\frac{\partial}{\partial z_i}\left[r(z)-\sum_{k=1}^nz_k\frac{\partial}{\partial z_k}r(z)\right]=0$ for each $i=1,\ldots, n$ which implies by Prop.\ \ref{prop:FieldOFConstantsForPartialDerivatives} that $r(z)-\sum_{k=1}^nz_k\frac{\partial}{\partial z_k}r(z)\in  \mathbb{F}(z^p)$ [resp., $r(z)-\sum_{k=1}^nz_k\frac{\partial}{\partial z_k}r(z)\in  \mathbb{F}(z^p)\cap \mathbb{F}(z)_1$ and note that either $p=0,1,$ or the nonzero homogeneous elements in $\mathbb{F}(z^p)$ must have degree divisible by $p$, in any case this implies $\mathbb{F}(z^p)\cap \mathbb{F}(z)_1=\{0\}$] from which it follows that $r(z)\in \mathbb{F}(z^p)+z_1\mathbb{F}(z^p)+\cdots+z_n\mathbb{F}(z^p)$ [resp., $r(z)\in z_1\mathbb{F}(z^p)\cap \mathbb{F}(z)_0+\cdots+z_n\mathbb{F}(z^p)\cap \mathbb{F}(z)_0$].
\end{proof}

The proof of this corollary also yields the following result, which will be needed later and may be of independent interest.
\begin{corollary}\label{cor:2ndMainCor}
    Let $\mathbb{F}$ be a field with $\operatorname{char}(\mathbb{F})=p$, $z_1,\ldots, z_n$ be indeterminates, and $r(z)\in \mathbb{F}(z)$ [where $z=(z_1,\ldots, z_n)$]. Then the following statements are true:
    \begin{itemize}
        \item[(i)] $r(z)\in \mathbb{F}(z^p)+z_1\mathbb{F}(z^p)+\cdots+z_n\mathbb{F}(z^p)$ if and only if $\frac{\partial}{\partial z_i}r(z)\in \mathbb{F}(z^p),\; i=1,\ldots, n$, in which case
    \begin{gather*}
        r(z)=r_0(z)+\sum_{i=1}^nz_i\frac{\partial}{\partial z_i}r(z),
    \end{gather*}
    for some $r_0(z)\in \mathbb{F}(z^p)$.
        \item[(ii)] $r(z)\in z_1\mathbb{F}(z^p)\cap\mathbb{F}(z)_0+\cdots+z_n\mathbb{F}(z^p)\cap\mathbb{F}(z)_0$ if and only if $r(z)\in \mathbb{F}(z)_1$ and $\frac{\partial}{\partial z_i}r(z)\in \mathbb{F}(z^p),\; i=1,\ldots, n$, in which case
    \begin{gather*}
        r(z)=\sum_{i=1}^nz_i\frac{\partial}{\partial z_i}r(z).
    \end{gather*}
        \item[(iii)] $\mathbb{F}(z^p)+z_1\mathbb{F}(z^p)+\cdots+z_n\mathbb{F}(z^p)$ is a vector space over the field $\mathbb{F}(z^p)$ with a basis $1,z_1,\ldots, z_n$.
        \item[(iv)] $z_1\mathbb{F}(z^p)\cap\mathbb{F}(z)_0+\cdots+z_n\mathbb{F}(z^p)\cap\mathbb{F}(z)_0$ is a vector space over the field $\mathbb{F}(z^p)\cap\mathbb{F}(z)_0$ with a basis $z_1,\ldots, z_n$. 
    \end{itemize}
\end{corollary}

\section{Proof of Theorem \ref{thm:MainResultSBRChar2ScalarCase}}\label{sec:PrfMainThmForSBRUsingDiffAlg}
In this section we prove Lemma \ref{lem:PartialDervZeroConnection2SBR} and then Theorem \ref{thm:MainResultSBRChar2ScalarCase}.  Let $\mathbb{F}$ be a field with $\operatorname{char}(\mathbb{F})=2$ and $z_1,\ldots, z_n$ be indeterminates. 

\begin{proof}[Proof of Lemma \ref{lem:PartialDervZeroConnection2SBR}]
    First, as $\mathbb{F}(z)$ is a field of characteristic $2$ then $(a+b)^2=a^2+b^2$ for all $a,b\in \mathbb{F}(z)$. From which it follows that $\operatorname{span}\{q^2:q\in \mathbb{F}(z)\}\subseteq\mathbb{F}(z^2)$. Second, for any $a\in \mathbb{F}(z),a\not =0$ we have $a^{-1}=(a^{-1})^2a$ from which it follows that $\operatorname{span}\{q^2:q\in \mathbb{F}(z)\}$ is a subfield of $\mathbb{F}(z^2)$. Finally, as $\operatorname{span}\{q^2:q\in \mathbb{F}(z)\}$ contains $z_1^2,\ldots, z_n^2$ it follows immediately from these facts that $\operatorname{span}\{q^2:q\in \mathbb{F}(z)\}=\mathbb{F}(z^2)$. 
    
    Now we will prove $\operatorname{span}\{q^2:q\in \mathbb{F}(z)_0\}=\mathbb{F}(z^2)_0$, where by definition $\mathbb{F}(z^2)_0=\mathbb{F}(z^2)\cap \mathbb{F}(z)_0$. First, $\mathbb{F}(z^2)_0$ is a subfield of $\mathbb{F}(z)$ since both $\mathbb{F}(z^2)$ and $\mathbb{F}(z)_0$ are subfields of $\mathbb{F}(z)$. And, as above, it follows that $\operatorname{span}\{q^2:q\in \mathbb{F}(z)_0\}$ is a subfield of $\mathbb{F}(z^2)_0$. It remains to prove $\mathbb{F}(z^2)_0\subseteq\operatorname{span}\{q^2:q\in \mathbb{F}(z)_0\}$. Let $r(z)\in \mathbb{F}(z^2)_0$. If $r(z)=0$ then we are done. Suppose $r(z)\not=0$. Then there exist $p(z),q(z)\in \mathbb{F}[z]$ with $p(z)q(z)\not=0$ and $\deg p(z)=\deg q(z)=m$ for some nonnegative integer $m$ such that $r(z)=\frac{p(z^2)}{q(z^2)}=\frac{z_n^{-2m}p(z^2)}{z_n^{-2m}q(z^2)}$. Clearly, $z_n^{-2m}p(z^2),z_n^{-2m}q(z^2)\in \operatorname{span}\{q^2:q\in \mathbb{F}(z)_0\}$ from which it follows that $r(z)\in \operatorname{span}\{q^2:q\in \mathbb{F}(z)_0\}$.

    Next, we will prove equality \eqref{AltReprSpanOfSquares} in Lemma \ref{lem:PartialDervZeroConnection2SBR}. First, as $\operatorname{span}\{q^2:q\in \mathbb{F}(z)\}=\mathbb{F}(z^2)$ then for each $i=1,\ldots, n$, we have $\operatorname{span}\{z_iq^2:q\in \mathbb{F}(z)\}=z_i\operatorname{span}\{q^2:q\in \mathbb{F}(z)\}=z_i\mathbb{F}(z^2)$ from which it follows that $\operatorname{span}\{q_0^2+z_1q_1^2+\cdots+z_nq_n^2:q_i\in \mathbb{F}(z)\}\subseteq\mathbb{F}(z^2)+z_1\mathbb{F}(z^2)+\cdots+z_n\mathbb{F}(z^2)$. Second, since $(a+b)^2=a^2+b^2$ for all $a,b\in \mathbb{F}(z)$ it follows that $\operatorname{span}\{q_0^2+z_1q_1^2+\cdots+z_nq_n^2:q_i\in \mathbb{F}(z)\}$ is a subspace of the vector space $\mathbb{F}(z)$ over the field $\mathbb{F}$ and clearly contains $q(z)^2$ and $z_iq(z)^2$ for any $q(z)\in \mathbb{F}(z)$ and each $i=1,\ldots,n$. Now, if $r(z)\in \mathbb{F}(z^2)+z_1\mathbb{F}(z^2)+\cdots+z_n\mathbb{F}(z^2)$ then $r(z)=r_0(z)+\sum_{i=1}^nz_ir_i(z)$ for some $r_i(z)\in \mathbb{F}(z^2)=\operatorname{span}\{q^2:q\in \mathbb{F}(z)\}$, $i=0,\ldots, n$. This implies that for each $i=0,\ldots, n$, $r_i(z)=\sum_{j=1}^{m_i}c_{ij}q_{ij}(z)^2$ for some $c_{ij}\in \mathbb{F}, q_{ij}(z)\in \mathbb{F}(z),j=1,\ldots, m_i$ so that $r(z)=\sum_{j=1}^{m_0}c_{0j}q_{0j}(z)^2+\sum_{i=1}^n\sum_{j=1}^{m_i}c_{ij}z_iq_{ij}(z)^2\in \operatorname{span}\{q_0^2+z_1q_1^2+\cdots+z_nq_n^2:q_i\in \mathbb{F}(z)\}$. This proves $\mathbb{F}(z^2)+z_1\mathbb{F}(z^2)+\cdots+z_n\mathbb{F}(z^2)\subseteq \operatorname{span}\{q_0^2+z_1q_1^2+\cdots+z_nq_n^2:q_i\in \mathbb{F}(z)\}$. Therefore, $\operatorname{span}\{q_0^2+z_1q_1^2+\cdots+z_nq_n^2:q_i\in \mathbb{F}(z)\}=\mathbb{F}(z^2)+z_1\mathbb{F}(z^2)+\cdots+z_n\mathbb{F}(z^2)$.

    Finally, we will prove equality \eqref{AltReprhSBRSpanOfSquares} in Lemma \ref{lem:PartialDervZeroConnection2SBR}. First, as $\operatorname{span}\{q^2:q\in \mathbb{F}(z)_0\}=\mathbb{F}(z^2)_0$ then for each $i=1,\ldots, n$, we have $\operatorname{span}\{z_iq^2:q\in \mathbb{F}(z)_0\}=z_i\operatorname{span}\{q^2:q\in \mathbb{F}(z)_0\}=z_i\mathbb{F}(z^2)_0$ from which it follows that $\operatorname{span}\{z_1q_1^2+\cdots+z_nq_n^2:q_i\in \mathbb{F}(z)_0\}\subseteq z_1\mathbb{F}(z^2)_0+\cdots+z_n\mathbb{F}(z^2)_0$. Second, as above, $\operatorname{span}\{z_1q_1^2+\cdots+z_nq_n^2:q_i\in \mathbb{F}(z)_0\}$ is a subspace of the vector space $\mathbb{F}(z)$ over the field $\mathbb{F}$ and clearly contains $z_iq(z)^2$ for any $q(z)\in \mathbb{F}(z)_0$ and each $i=1,\ldots,n$. Now, if $r(z)\in z_1\mathbb{F}(z^2)_0+\cdots+z_n\mathbb{F}(z^2)_0$ then $r(z)=\sum_{i=1}^nz_ir_i(z)$ for some $r_i(z)\in \mathbb{F}(z^2)_0=\operatorname{span}\{q^2:q\in \mathbb{F}(z)_0\}$, $i=1,\ldots, n$. This implies that for each $i=1,\ldots, n$, $r_i(z)=\sum_{j=1}^{m_i}c_{ij}q_{ij}(z)^2$ for some $c_{ij}\in \mathbb{F}, q_{ij}(z)\in \mathbb{F}(z)_0,j=1,\ldots, m_i$ so that $r(z)=\sum_{i=1}^n\sum_{j=1}^{m_i}c_{ij}z_iq_{ij}(z)^2\in \operatorname{span}\{z_1q_1^2+\cdots+z_nq_n^2:q_i\in \mathbb{F}(z)_0\}$. This proves $z_1\mathbb{F}(z^2)_0+\cdots+z_n\mathbb{F}(z^2)_0\subseteq \operatorname{span}\{z_1q_1^2+\cdots+z_nq_n^2:q_i\in \mathbb{F}(z)_0\}$. Therefore, $\operatorname{span}\{z_1q_1^2+\cdots+z_nq_n^2:q_i\in \mathbb{F}(z)_0\}=z_1\mathbb{F}(z^2)_0+\cdots+z_n\mathbb{F}(z^2)_0$.
\end{proof}

\begin{proof}[Proof of Theorem \ref{thm:MainResultSBRChar2ScalarCase}]
(i): Suppose that $F(z)\in\mathbb{F}(z)^{1\times 1}$ [$z=(z_1,\ldots, z_n)$] has an SBR. Then
    \begin{gather*}
    F(z)=A(z)/A_{22}(z),\\
    A(z)=[A_{ij}(z)]_{i,j=1,2}=A_0+ z_1A_1+ \cdots +z_nA_n,\\
    A_j^T=A_j\in \mathbb{F}^{m\times m},\;j=0,\ldots ,n.
\end{gather*}
Our goal is to prove \eqref{SBRSubspaceChar2} or, equivalently, by \eqref{AltReprSpanOfSquares}:
\begin{gather}
    F(z)\in\mathbb{F}(z^2)+z_1\mathbb{F}(z^2)+\cdots+z_n\mathbb{F}(z^2).\label{AltInclusionToProveForMainResult}
\end{gather}

First, if $F(z)=0$ then  \eqref{AltInclusionToProveForMainResult} is clearly true. Hence, assume that $F(z)\not=0$. Then $0\not=\det F(z) \det A_{22}(z)= \det A(z)$ by the Schur determinantal formula \cite[Lemma 1.1]{Elsinger:BRS:2026} and so by the Banachiewicz inversion formula \cite[Lemma 2.6]{Elsinger:BRS:2026} we have 
\begin{align}\label{Equationfz}
    F(z) = 
        \left\{\begin{bmatrix}
            A(z)^{-1}
        \end{bmatrix}_{11}\right\}^{-1}.
\end{align}
 Then for each $i=1,\ldots, n$ we have by Lemma \ref{lem:BasicPropertiesOfDerivations}(ii) and representation \eqref{Equationfz} that
 \begin{align*}
        \frac{\partial F(z)}{\partial z_i} &=  - \left\{\begin{bmatrix}
            A(z)^{-1}
        \end{bmatrix}_{11}\right\}^{-1} \frac{\partial \begin{bmatrix}
            A(z)^{-1}
        \end{bmatrix}_{11}}{\partial z_i} \left\{\begin{bmatrix}
            A(z)^{-1}
        \end{bmatrix}_{11}\right\}^{-1} \\
        &=  -F(z)\frac{\partial \begin{bmatrix}
            A(z)^{-1}
        \end{bmatrix}_{11}}{\partial z_i} F(z)\\
        &=  -F(z) \begin{bmatrix} \frac{ \partial
            A(z)^{-1}}{\partial z_i}
        \end{bmatrix}_{11} F(z)\\
         &=  -F(z) \begin{bmatrix}  -A(z)^{-1} \frac{ \partial
          A(z) }{\partial z_i}  A(z)^{-1}
        \end{bmatrix}_{11} F(z)\\
         &=  F(z)^2 \begin{bmatrix}  A(z)^{-1} A_i  A(z)^{-1}
        \end{bmatrix}_{11}\\
        &= F(z)^2 \begin{bmatrix}  (A(z)^{-1})^T A_i  A(z)^{-1}
        \end{bmatrix}_{11}.
\end{align*}
Next, we claim that 
\begin{gather*}
    \frac{\partial F(z)}{\partial z_i}\in \mathbb{F}(z^2),
\end{gather*}
for $i=1,\ldots, n$. Once this claim is proven, it will follow immediately from
Corollary \ref{cor:2ndMainCor} that \eqref{AltInclusionToProveForMainResult} holds thereby proving Theorem \ref{thm:MainResultSBRChar2ScalarCase}(i).

Before we prove this claim recall the following definition from \cite{Albert:1938:SAM}.
\begin{definition}
    Let $A$ be a symmetric matrix with entries in a field of characteristic $2$. Then $A$ is called alternate if all its diagonal entries are zero; otherwise, it is called non-alternate symmetric.
\end{definition}

\begin{proof}[Proof of the claim] 
First, by the hypotheses, both $\mathbb{F}$ and hence $\mathbb{F}(z)$ are fields of characteristic $2$ and $A_l$ is an $m\times m$ symmetric matrix with entries in $\mathbb{F}\subseteq \mathbb{F}(z)$, for each $l=1,\ldots,n$. For any such $l$, we will consider separately the two possible cases for $A_l$: it is either alternate or non-alternate symmetric. Suppose $A_l$ is alternate. Then by \cite[Theorem 1]{Albert:1938:SAM}, so is the symmetric matrix $(A(z)^{-1})^T A_l  A(z)^{-1}$ by congruency. Hence for the alternate matrix $(A(z)^{-1})^T A_l  A(z)^{-1}$, all its diagonal entries are zero and, in particular, $\begin{bmatrix}  (A(z)^{-1})^T A_l  A(z)^{-1}
        \end{bmatrix}_{11}=0$ implying $\frac{\partial F(z)}{\partial z_l} = F(z)^2 \begin{bmatrix}  (A(z)^{-1})^T A_l  A(z)^{-1}
        \end{bmatrix}_{11}=0\in \mathbb{F}(z^2)$. Suppose  now that $A_l$ is a non-alternate symmetric matrix. In this case, \cite[Theorem 6]{Albert:1938:SAM}, says $A_l$ is congruent to a diagonal matrix, that is, there exists an invertible matrix $P\in \mathbb{F}^{m\times m}$ and scalars $b_1,\ldots, b_m\in \mathbb{F}$ (with dependence on index $l$ suppressed) such that
\begin{equation*}
    A_l=P^T~ \text{diag}(b_1,\ldots,b_m)~ P.
\end{equation*}
Thus in terms of the entries of the matrix $PA(z)^{-1}=\begin{bmatrix}
    c_{ij}(z)
\end{bmatrix}_{i,j=1,\ldots, m}$, we have
\begin{align*}
 \frac{\partial F(z)}{\partial z_l} &=  F(z)^2 \begin{bmatrix}  (A(z)^{-1})^T A_l  A(z)^{-1}
        \end{bmatrix}_{11} \\
        &=F(z)^2 \begin{bmatrix}  (PA(z)^{-1})^T ~\text{diag}(b_1,\ldots,b_m) ~PA(z)^{-1}
        \end{bmatrix}_{11} \\
        &=\sum_{i=1}^{m} b_i(F(z)c_{i1}(z))^2\in \mathbb{F}(z^2).
\end{align*}
This completes the proof of the claim.
\end{proof}

(ii): Suppose $F(z)\in \mathbb{F}(z)^{1\times 1}$ has an hSBR. Then $F(z)\in \mathbb{F}(z)_1$, which implies that $\frac{\partial}{\partial z_i}F(z)\in \mathbb{F}(z)_0$ for each $i=1,\ldots, n$, and since $F(z)$ has an SBR it follows by part (i) that $\frac{\partial}{\partial z_i}F(z)\in \mathbb{F}(z^2)\cap \mathbb{F}(z)_0=\mathbb{F}(z^2)_0$ for each $i=1,\ldots, n$ implying by Corollary \ref{cor:2ndMainCor} that $F(z)\in z_1\mathbb{F}(z^2)_0+\cdots+z_n\mathbb{F}(z^2)_0$ which by  \eqref{AltReprhSBRSpanOfSquares} proves Theorem \ref{thm:MainResultSBRChar2ScalarCase}(ii).
\end{proof}

\section{Schur Complement Algebra}\label{sec:SchurComplAlgebra}
In our discussion below, $\mathbb{F}$ is a field and $z_1,\ldots, z_n,z_{n+1}$ are indeterminates with $z=(z_1,\ldots,z_n)$ and $z/z_{n+1}=(z_1/z_{n+1},\ldots, z_n/z_{n+1})$. Assume that $A=A(z), B=B(z)$ are linear matrix pencils with entries in $\mathbb{F}[z]$ and that there exist partitions of the matrices $A=[A_{ij}]_{i,j=1,2}, B=[B_{ij}]_{i,j=1,2}$ into $2\times 2$ block matrix form, such that $A_{22}, B_{22}$ are invertible. In addition, we assume their Schur complements $A/A_{22}, B/B_{22}$ are the same size (or, equivalently, $A_{11}, B_{11}$ are the same size); so, $A/A_{22},B/B_{22}\in \mathbb{F}(z)^{k\times k}$ (for some $k\in\mathbb{N}$). Let $U\in \mathbb{F}^{l\times k}, V\in \mathbb{F}^{k\times l},$ and $X=X(z)\in \mathbb{F}(z)^{k\times k}$ be an invertible linear matrix pencil.

It was shown in \cite{Elsinger:BRS:2026} that for each of the following operations (I)-(VI), the output has a BR over the field $\mathbb{F}(z)$ and for (VII), the output has a BR over the field $\mathbb{F}(z_1,\ldots, z_n,z_{n+1})$:
\begin{itemize}
           \item[(I)] Scalar multiplication (\cite[Lemma 2.1]{Elsinger:BRS:2026}): $\lambda (A/A_{22})$, where $\lambda\in \mathbb{F}$;
           \item[(II)] Sum (\cite[Lemma 2.2]{Elsinger:BRS:2026}): $A/A_{22}+B/B_{22}$;
            \item[(III)] Symmetrization (\cite[Lemma 2.3]{Elsinger:BRS:2026}): $A/A_{22}+(A/A_{22})^T$;
           \item[(IV)] Matrix multiplication (\cite[Lemma 2.4 \& 2.5]{Elsinger:BRS:2026}): $$(a)\;U(A/A_{22})V;\; (b)\;(A/A_{22})X^{-1}(B/B_{22});$$
           \item[(V)] Inversion (\cite[Lemma 2.6]{Elsinger:BRS:2026}): $(A/A_{22})^{-1}$, if $A/A_{22}$ is invertible;
            \item[(VI)] Kronecker product (\cite[Lemma 2.7]{Elsinger:BRS:2026}): $A/A_{22}\otimes I$;
           \item[(VII)] Homogenization (\cite[Lemma 2.8]{Elsinger:BRS:2026}): $z_{n+1}A(z/z_{n+1})/A_{22}(z/z_{n+1})$.
\end{itemize}
In addition, for these operations the corresponding lemmas in \cite{Elsinger:BRS:2026} give us even more regarding the output BR that can be produced from the input regarding sBR, hBR, or hSBR: For operations (I), (II), and (VI) if $A(z)$ and $B(z)$ are both sLP, hLP, or hsLP then the output has a SBR, hBR, or hSBR, respectively. For operation (III), the output has an SBR and if $A(z)$ is an hLP then the output has an hSBR. For operation (IV)(a), if $A(z)$ is hLP then the output has an hBR. Furthermore, if $V=U^T$ and $A(z)$ is sLP or hsLP then the output has an SBR or hSBR, respectively. For operation (IV)(b), if $A(z), B(z), X(z)$ are hLP then the output has an hBR. Furthermore, if $B(z)=A(z)^T,X(z)=X(z)^T$ then the output has an SBR and if, in addition, $A(z)$ and $X(z)$ are hLP then the output has an hSBR. For operation (V), if $A(z)$ is sLP then the output has an SBR. Finally, for operation (VII), the output has an hBR and if $A(z)$ is a sLP then the output has an hSBR. 

Now from those results and since $M=\begin{bmatrix} M & 0\\ 0 & 1 \end{bmatrix}/\begin{bmatrix} 1 \end{bmatrix}$ for any matrix $M\in \mathbb{F}^{k\times k}$ and  $z_i=\begin{bmatrix} z_i & 0\\ 0 & 1 \end{bmatrix}/\begin{bmatrix} 1 \end{bmatrix}$ for $i=1,\ldots, n$, it is easily shown, as in \cite[Sec.\ 2.4]{Elsinger:BRS:2026}, that Theorem \ref{thm:BessmertnyiRealizabilityThm} is true with the exception of case (iii) in statements (c) and (d) of that theorem. For this, there are three results that follow almost immediately from all this that are worth noting. The first is \cite[Lemma 2.10 \& Lemma 5.1]{Elsinger:BRS:2026}:

\begin{lemma}\label{lem:SBRhSBRProofReducesToDiagonals}
    Suppose $F(z)=[F_{ij}(z)]_{i,j=1,\ldots, k}\in \mathbb{F}(z)^{k\times k}$. Then $F(z)$ has an SBR (hSBR) if and only if its $i$th diagonal entry $F_{ii}(z)\in \mathbb{F}(z)^{1\times 1}$ has an SBR (hSBR) for each $i=1,\ldots, k$.
\end{lemma}

The next two results require a few preliminary definitions from \cite[Sec.\ 4]{Elsinger:BRS:2026} and so we do so in the next subsection.

\subsection{Jordan Algebra of SBRs and hSBRs}

In this section, we give results that help to explain the discrepancy in classifying SBRs (and hSBRs) over characteristic $2$. The crucial construction is that of a Jordan Algebra. For the standard definitions and notions below we follow \cite{McCrimmon:1969:OHT, McCrimmon:2010:TJA}.

Consider the associative algebra $\mathcal{A}=\mathbb{F}(z)^{k\times k}$ over the field $\mathbb{F}$, where $\operatorname{char}(\mathbb{F})=2$. As is well-known, since it is an associative algebra, we can form a unital quadratic Jordan algebra $\mathcal{A}^+$ via
\begin{gather*}
    \mathcal{A}^+:U_xy=xyx,\;x^2=xx,\;\{xyz\}=xyz+zyx,\;x\circ y=xy+yx.
\end{gather*}
An important example of a Jordan subalgebra of $\mathcal{A}^+$ (and hence by definition a special Jordan algebra) is the symmetric elements with transpose involution (also known as the Hermitian algebra of self-adjoint elements in the associative algebra $\mathcal{A}$ with involution $T$):
\begin{gather*}
    H(\mathcal{A},T)=\{x\in \mathcal{A}:x^T=x\}\subseteq \mathcal{A}^+.
\end{gather*}
We recall the following definition (more specifically, it is called an ample Hermitian algebra or ample Hermitian subspace, and it is also a Jordan algebra).
\begin{definition}
    A set $H_0(\mathcal{A},T)$ is called an ample subspace of $H(\mathcal{A},T)$ if $H_0(\mathcal{A},T)$ is a subspace of $H(\mathcal{A},T)$ that contains all ``traces" $x+x^T,$ all ``norms" $xx^T$, and $xH_0(\mathcal{A},T)x^T\subseteq H_0(\mathcal{A},T)$ for all elements $x\in \mathcal{A}$.
\end{definition}

Now the largest ample subspace is $H(\mathcal{A},T)$, while the smallest is the \textit{core} of $\mathcal{A}$ which is the space $K_0(\mathcal{A})$ spanned by the norms and traces, i.e.,
$$K_0(\mathcal{A}):=\bigg\{x+x^T+\sum_i \lambda_ix_ix_i^T~\bigg|~x,x_i\in\mathcal{A},\lambda_i\in\mathbb{F}\bigg\},$$
and the \textit{ample hull} of a subset $S\subseteq H(\mathcal{A},T)$ is the smallest ample subspace containing $S$ (cf.\ \cite[p.\ 387]{McCrimmon:1969:OHT}).

This second result of note is the following.
\begin{theorem}\label{thm:JordanAlgSBRs}
    Consider the associative algebra $\mathcal{A}=\mathbb{F}(z)^{k\times k}$ over a field $\mathbb{F}$ of characteristic $2$ with matrix transposition $T$ acting as the involution, giving rise to the unital quadratic Jordan algebra $\mathcal{A}^+$. Define the Jordan subalgebras, 
    \begin{gather*}
        H(\mathcal{A},T)=\{F(z)\in \mathbb{F}(z)^{k\times k}:F(z)^T=F(z)\},\\
        H_{SBR}(\mathcal{A},T)=\{F(z)\in \mathbb{F}(z)^{k\times k}:F(z) \text{ has an SBR}\}.
    \end{gather*}
    Then $H_{SBR}(\mathcal{A},T)$ is an ample subspace of $H(\mathcal{A},T)$, in fact, it is the ample hull of $\{I_k,z_1I_k,\ldots, z_nI_k\}\subseteq H(\mathcal{A},T)$.
\end{theorem}
\begin{proof}
    Ampleness of $H_{SBR}(\mathcal{A},T)$ was already established in \cite[Theorem 4.2]{Elsinger:BRS:2026}.  Now it is clear that $S:=\{I_k,z_1I_k,\ldots, z_nI_k\}\subseteq H(\mathcal{A},T)$. Let us consider an ample subspace $\mathcal{H}$ containing $S$. We claim that $H_{SBR}(\mathcal{A},T)\subseteq \mathcal{H}$, and to that effect, choose $h\in H_{SBR}(\mathcal{A},T)$. Since $\mathcal{H}$ contains all traces, it is enough to show that $\tilde{h}:=\text{diag} (h):=\text{diag} (h_1,h_2,\ldots,h_k)\in\mathcal{H}$ and, as $\mathcal{H}$ is a linear subspace, this amounts to showing that $h_le_le_l^T=\text{diag} (0,\ldots, h_l,\ldots,0)\in \mathcal{H}$, where $e_l\in \mathbb{F}^k$ is the $l$th standard basis (column) vector, $l=1,\ldots, k$. As $h$ has an SBR, then by Lemma \ref{lem:SBRhSBRProofReducesToDiagonals} and Theorem \ref{thm:MainResultSBRChar2ScalarCase}, each $h_l$ can be written as a finite sum of terms of the form $z_i \lambda_{lij} r_{lij}^2$ (setting $z_0:=1)$ for some $\lambda_{lij}\in \mathbb{F}, r_{lij}\in \mathbb{F}(z)$, hence by linearity we may reduce the task to showing that $\text{diag}(0,0,\ldots,z_i r_{lij}^2,\ldots 0)=z_i r_{lij}^2e_le_l^T=(r_{lij}e_le_l^T) z_iI_k (r_{lij}e_le_l^T)^T\in\mathcal{H}$; however, this follows from ampleness of $\mathcal{H}$.
\end{proof}

A new result that follows immediately from all this and \cite[Lemma 5.2]{Elsinger:BRS:2026} will be worth noting as well.
\begin{theorem}\label{thm:JordanAlghSBRs}
    Consider the associative algebra $\mathcal{A}_0=\mathbb{F}(z)_0^{k\times k}$ over a field $\mathbb{F}$ of characteristic $2$ with matrix transposition $T$ acting as the involution, giving rise to the unital quadratic Jordan algebra $\mathcal{A}_0^+$ (i.e., the same operations as $\mathcal{A}^+$, but restricted to elements from $\mathcal{A}_0$ instead of $\mathcal{A}$). Define the Jordan subalgebras, 
    \begin{gather*}
        H(\mathcal{A}_0,T)=\{F(z)\in \mathbb{F}(z)_0^{k\times k}:F(z)^T=F(z)\},\\
        H_{hSBR}(\mathcal{A}_0,T)=\{F(z/z_n):F(z)\in \mathbb{F}(z)^{k\times k} \text{ has an hSBR}\}.
    \end{gather*}
    Then $H_{hSBR}(\mathcal{A}_0,T)$ is an ample subspace of $H(\mathcal{A}_0,T)$, in fact, it is the ample hull of $\{(z_1/z_n)I_k,\ldots, (z_{n-1}/z_n)I_k,I_k\}\subseteq H(\mathcal{A}_0,T)$. Moreover,
    \begin{gather*}
        z_n H_{hSBR}(\mathcal{A}_0,T)=\{F(z)\in \mathbb{F}(z)^{k\times k}:F(z) \text{ has an hSBR}\}.
    \end{gather*}
\end{theorem}
\begin{proof}
    Using the arguments presented in proving \cite[Lemma 5.2]{Elsinger:BRS:2026} and following the proof of Theorem \ref{thm:JordanAlgSBRs} above, but replacing the field $\mathbb{F}(z)$ by $\mathbb{F}(z)_0$, proves immediately the theorem.
\end{proof} 

\section{Proof of Theorem \ref{thm:BessmertnyiRealizabilityThm}}\label{sec:CompletingPrfMainThm}
We are now ready to complete the proof of Theorem \ref{thm:BessmertnyiRealizabilityThm}. From the results of Sec.\ \ref{sec:SchurComplAlgebra}, it remains only to prove statements (c)(iii) and (d)(iii) of Theorem \ref{thm:BessmertnyiRealizabilityThm}. Now by Lemma \ref{lem:SBRhSBRProofReducesToDiagonals} we need only prove these statements when $\mathbb{F}$ is a field with $\operatorname{char}(\mathbb{F})=2$ and $k=1$. By Theorem \ref{thm:JordanAlgSBRs} it follows immediately that $\operatorname{span}\{q_0^2+z_1q_1^2+\cdots+z_nq_n^2:q_i\in \mathbb{F}(z)\}\subseteq \{F(z)\in \mathbb{F}(z)^{1\times 1}:F(z)\text{ has an SBR}\}$. Conversely, by Theorem \ref{thm:MainResultSBRChar2ScalarCase} we have $\{F(z)\in \mathbb{F}(z)^{1\times 1}:F(z)\text{ has an SBR}\}\subseteq \operatorname{span}\{q_0^2+z_1q_1^2+\cdots+z_nq_n^2:q_i\in \mathbb{F}(z)\}$. This proves statement (c)(iii) of Theorem \ref{thm:BessmertnyiRealizabilityThm}. Similarly, by Theorem \ref{thm:JordanAlghSBRs} it follows immediately that $\operatorname{span}\{z_1q_1^2+\cdots+z_nq_n^2:q_i\in \mathbb{F}(z)_0\}\subseteq \{F(z)\in\mathbb{F}(z)^{1\times 1}:F(z)\text{ has an hSBR}\}$. Conversely, by Theorem \ref{thm:MainResultSBRChar2ScalarCase} we have $\{F(z)\in \mathbb{F}(z)^{1\times 1}:F(z)\text{ has an hSBR}\}\subseteq \operatorname{span}\{z_1q_1^2+\cdots+z_nq_n^2:q_i\in \mathbb{F}(z)_0\}$. This proves statement (d)(iii) of Theorem \ref{thm:BessmertnyiRealizabilityThm}. This completes the proof of Theorem \ref{thm:BessmertnyiRealizabilityThm}.

\section{Proof of Theorem \ref{thm:MainResultSBRFieldExtension}}\label{sec:PrfThmSBRFieldExt}
In this section we prove Theorem \ref{thm:MainResultSBRFieldExtension}. Let $\mathbb{F}$ be a field, $\mathbb{K}$ be a field extension of $\mathbb{F}$, and $z_1,\ldots, z_n, \dots, z_{N}$ be indeterminates with $N\geq n$. In the case $\operatorname{char}(\mathbb{F})\not=2$, it follows immediately by Theorem \ref{thm:BessmertnyiRealizabilityThm} that Theorem \ref{thm:MainResultSBRFieldExtension} is true. Hence assume that $\operatorname{char}(\mathbb{F})=2$. Suppose that $F(z)\in\mathbb{F}(z)^{k\times k}$ [$z=(z_1,\ldots, z_n)$]. Then it follows by Theorem \ref{thm:BessmertnyiRealizabilityThm} that if $F(z)$ has a BR or hBR over the field $\mathbb{K}(w)$, where $w=(z_1,\dots, z_N)$, then it has a BR or hBR, respectively, over the field $\mathbb{F}(z)$.  Hence, suppose that $F(z)$ has a SBR or hSBR over the field $\mathbb{K}(w)$. Then by Lemma \ref{lem:SBRhSBRProofReducesToDiagonals} it suffices to prove Theorem \ref{thm:MainResultSBRFieldExtension} assuming $k=1$.
Suppose $F(z)\in \mathbb{F}(z)^{1\times 1}$ has an SBR (resp., hSBR) over the field $\mathbb{K}(w)$. Then by Theorem \ref{thm:MainResultSBRChar2ScalarCase} we have $F(z)\in \mathbb{K}(w^2)+w_1\mathbb{K}(w^2)+\cdots+w_N\mathbb{K}(w^2)$ [resp., $F(z)\in w_1\mathbb{K}(w^2)_0+\cdots+w_N\mathbb{K}(w^2)_0$] and therefore by Corollary \ref{cor:MainCor} we have $F(z)\in \mathbb{F}(z^2)+z_1\mathbb{F}(z^2)+\cdots+z_n\mathbb{F}(z^2)$ [resp., $F(z)\in z_1\mathbb{F}(z^2)_0+\cdots+z_n\mathbb{F}(z^2)_0$] and hence by Theorem \ref{thm:JordanAlgSBRs} (resp., Theorem \ref{thm:JordanAlghSBRs}) it follows that $F(z)$ has an SBR (resp., hSBR) over the field $\mathbb{F}(z)$. This completes the proof of Theorem \ref{thm:MainResultSBRFieldExtension}.

\section*{Acknowledgement}
AW is grateful to the National Science Foundation for support through grant DMS-2410678 and to the Simons Foundation for the support through grant MPS-TSM-00002799. 

\section*{Appendix: Auxiliary Proofs and Discussions}\label{sec:AppendixAuxPrfs}

Let $\mathbb{F}$ be a field with $\operatorname{char}(\mathbb{F})=2$ and $z_1,\ldots, z_n$ be $n$ indeterminates. For each $n$-tuple $\alpha=(\alpha_1,\ldots, \alpha_n)\in (\mathbb{N}\cup\{0\})^n$, define the monomial $z^{\alpha}=\prod_{i=1}^n z_i^{\alpha_i}\in \mathbb{F}(z)$ [where $z=(z_1,\ldots, z_n)$]. 

We begin by proving (following the proof of Prop.\ 2.2 in \cite{Santos:2023:PCD}) that $\mathbb{F}(z)$ is a $2^n$-dimensional vector space over $\mathbb{F}(z^2)$. The result is immediate once we prove the claim: the monomials of the form $\prod _{i=1}^nz_i^{\alpha_i}$ with $\alpha_i\in\{0,1\}$ are a basis for the vector space $\mathbb{F}(z)$ over the field $\mathbb{F}(z^2)$. First, it is clear that these monomials are linearly independent over the field $\mathbb{F}(z^2)$. Second, let $p(z)/q(z)\in \mathbb{F}(z)$ with $p(z),q(z)\in \mathbb{F}[z]$ and $q(z)\not= 0$. Then 
\begin{gather*}
    \frac{p(z)}{q(z)}=\frac{1}{q(z)^2}p(z)q(z)
\end{gather*}
Next, we can write $p(z)q(z)=\sum_{\alpha\in (\mathbb{N}\cup\{0\})^n}c_{\alpha}z^{\alpha}$ for some scalars $c_{\alpha}\in \mathbb{F}$ that are zero for all but a finite number of $\alpha$, hence we have $\alpha_i=2m_i+\beta_i$ for some $m_i\in \mathbb{N}\cup\{0\},\beta_i\in \{0,1\}$ and so
\begin{gather*}
    p(z)q(z)=\sum_{\alpha\in (\mathbb{N}\cup\{0\})^n}c_{\alpha}z^{\alpha}=\sum_{\alpha\in (\mathbb{N}\cup\{0\})^n}c_{\alpha}\left(\prod _{i=1}^nz_i^{m_i}\right)^2\prod _{i=1}^nz_i^{\beta_i}.
\end{gather*}
It follows from this and $\frac{1}{q(z)^2}\in \mathbb{F}(z^2)$ that $p(z)/q(z)$ is a linear combination, with coefficients in the field $\mathbb{F}(z^2)$, of the set of monomials $\{\prod _{i=1}^nz_i^{\alpha_i}:\alpha_i\in \{0,1\}, i=1,\ldots, n\}$. This proves the claim.

Next, we will prove that $\mathbb{F}(z)_1$ is a $2^{n-1}$-dimensional vector space over $\mathbb{F}(z^2)_0$. We begin by finding a basis of $\mathbb{F}(z)_0$ over $\mathbb{F}(z^2)_0$. First, if $r(z)\in \mathbb{F}(z)_0$ then by definition it is a homogeneous degree-$0$ rational function so that $r(z)=r(z_1/z_n,\ldots, z_n/z_n)$. It follows that the function $\varphi:\mathbb{F}(z)_0\rightarrow \mathbb{F}(y)$ [where $y=(y_1,\ldots, y_{n-1})$ are $n-1$ indeterminates] defined by $\varphi(r(z))=r(y_1,\ldots, y_{n-1},1)$ for $r(z)\in \mathbb{F}(z)_0$, is a field isomorphism. In particular, $\varphi (z_i/z_n)=y_i$ for $i=1,\ldots, n$ (setting $y_n:=1)$. As $\varphi$ sends $\mathbb{F}(z)_0$ and $\mathbb{F}(z^2)_0$ to $\mathbb{F}(y)$ and $\mathbb{F}(y^2)$, respectively, then the dimension of $\mathbb{F}(z)_0$ over $\mathbb{F}(z^2)_0$ is the same as that of $\mathbb{F}(y)$ over $\mathbb{F}(y^2)$, which by our proof above is precisely $2^{n-1}$ in which a basis for $\mathbb{F}(y)$ over $\mathbb{F}(y^2)$ is given by the monomials $\{\prod _{i=1}^{n-1}y_i^{\alpha_i}:\alpha_i\in \{0,1\}, i=1,\ldots, n-1\}$ (in the case $n=1$ it is just $\{1\}$) and therefore a basis for $\mathbb{F}(z)_0$ over $\mathbb{F}(z^2)_0$ is the preimage under $\varphi$ of those monomials which is $\beta:=\varphi^{-1}(\{\prod _{i=1}^{n-1}y_i^{\alpha_i}:\alpha_i\in \{0,1\}, i=1,\ldots, n-1\})=\{z_n^{-\sum_{i=1}^{n-1}\alpha_i}\prod _{i=1}^{n-1}z_i^{\alpha_i}:\alpha_i\in \{0,1\}, i=1,\ldots, n-1\}$ (in the case $n=1$ it is just $\{1\}$). Next, to prove that $\mathbb{F}(z)_1$ is a $2^{n-1}$-dimensional vector space over $\mathbb{F}(z^2)_0$ it suffices to show that $z_n\beta=\{z_n^{1-\sum_{i=1}^{n-1}\alpha_i}\prod _{i=1}^{n-1}z_i^{\alpha_i}:\alpha_i\in \{0,1\}, i=1,\ldots, n-1\}$ (in the case $n=1$ it is just $\{z_n\}$) is a basis for $\mathbb{F}(z)_1$ over $\mathbb{F}(z^2)_0$. First, it is clear that $z_n\beta\subseteq \mathbb{F}(z)_1$ and is linearly independent over $\mathbb{F}(z^2)_0$ since $\beta$ is. Next, let $r(z)\in \mathbb{F}(z)_1$. Then $r(z)/z_n\in \mathbb{F}(z)_0$ which implies it is a linear combination of the basis $\beta$ with coefficients in $\mathbb{F}(z^2)_0$ and hence it follows that $r(z)$ is a linear combination of elements from $z_n\beta$ with coefficients in $\mathbb{F}(z^2)_0$. This proves the claim which completes the proof.

Finally, we want to conclude this appendix by an expanded discussion on some technicalities regarding Theorem \ref{thm:MainResultSBRFieldExtension}.  Let $\mathbb{K}$ be a field extension of $\mathbb{F}$. We will address how the differential field structure of $\mathbb{F}$, inherited from the natural inclusion $\mathbb{F}\subseteq \mathbb{F}(z)$, lifts to that of $\mathbb{K}$, thereby making it the field of constants of the (extended) derivations $\partial/\partial z_1,\ldots,\partial/\partial z_n$ and allowing us to use Prop.\  \ref{prop:FieldOFConstantsForPartialDerivatives}.  We begin by constructing the algebraic closure of $\mathbb{F}$ within $\mathbb{K}$, say $\bar{\mathbb{F}}$. Let $\alpha \in \bar{\mathbb{F}}$. If the minimal polynomial of $\alpha$, say $m_{\alpha}\in \mathbb{F}[x]$ is separable, observe that
$\partial m_{\alpha} (\alpha)/\partial z_i=0\Rightarrow \partial \alpha/\partial z_i \times m'_{\alpha}(\alpha)=0.$
Since $m'_{\alpha}(\alpha)\neq 0$ due to separability one necessarily has $\partial \alpha/\partial z_i=0$. Thus, we can canonically extend the derivations on $\mathbb{F}$ to all separable elements over $\mathbb{F}$ by simply expanding the field of constants. On the other hand, if $m_\alpha$ is inseparable over $\mathbb{F}$, by replacing $\mathbb{F}$ with a superfield containing all separable elements over $\mathbb{F}$ if necessary and abuse of notation, we can assume that $\alpha$ is purely inseparable over $\mathbb{F}$. In this case, there is unfortunately no canonical extension of the derivations, however, for convenience, we choose to define  $\partial \alpha/\partial z_i:=0$ for all such elements and thus, $\bar{\mathbb{F}}$ becomes the field of constants of the extended derivations. Finally, since $\mathbb{K}/\bar{\mathbb{F}}$ is generated by purely transcendental elements, we may set their derivatives to $0$, thus settling our claim. As an example, one may consider $\mathbb{F}:=\mathbb{Q}\subseteq \mathbb{Q}(z)$ and $\mathbb{K}:=\mathbb{Q}(\sqrt{2}, t)\subseteq \mathbb{Q}(\sqrt{2},t)(z,w)$ where $t$ is an indeterminate. We extend $d/dz$ to $\mathbb{K}(z,w)$ such that $\mathbb{K}$ is the field of constants of $d/dz$ by declaring that $d\sqrt{2}/dz=dt/dz=0$.

\printbibliography

@article{Albert:1938:SAM,
 author = {A. Adrian Albert},
 journal = {Transactions of the American Mathematical Society},
 number = {3},
 pages = {386--436},
 publisher = {American Mathematical Society},
 title = {Symmetric and Alternate Matrices in An Arbitrary Field, {I}},
 volume = {43},
 year = {1938},
 doi = {10.2307/1990068}
}

@book{Cox:2015:IVA,
  title={Ideals, Varieties, and Algorithms: An Introduction to Computational Algebraic Geometry and Commutative Algebra},
  author={Cox, D.A. and Little, J. and O'Shea, D.},
  series={Undergraduate Texts in Mathematics},
  year={2015},
  publisher={Springer},
  edition = {4th Edition}
}

@book{Dummit:2004:AAL,
  title={Abstract Algebra},
  author={Dummit, D.S. and Foote, R.M.},
  year={2004},
  edition={3rd Edition},
  publisher={Wiley}
}

@article{Elsinger:BRS:2026,
title = {Bessmertnyĭ realizations of symmetric multivariate rational matrix functions over any field},
journal = {Linear Algebra and its Applications},
volume = {737},
pages = {80-118},
year = {2026},
doi = {10.1016/j.laa.2026.02.014},
author = {Jason Elsinger and Ian Orzel and Aaron Welters},
}

@book{Kaplansky:1957:IDA,
  title={An Introduction to Differential Algebra},
  author={Kaplansky, I.},
  series={Actualit{\'e}s scientifiques et industrielles},
  year={1957},
  publisher={Hermann}
}

@book{Kolchin:1973:DAA,
  title={Differential Algebra and Algebraic Groups},
  author={Kolchin, E.R.},
  series={Pure and Applied Mathematics},
  year={1973},
  publisher={Academic Press}
}

@article{McCrimmon:1969:OHT,
  author  = {McCrimmon, Kevin},
  title   = {On Herstein's theorems relating Jordan and associative algebras},
  journal = {Journal of Algebra},
  volume  = {13},
  number  = {3},
  pages   = {382--392},
  year    = {1969},
  doi     = {10.1016/0021-8693(69)90081-7}
}

@book{McCrimmon:2010:TJA,
  author    = {McCrimmon, Kevin},
  title     = {A Taste of {J}ordan Algebras},
  series    = {Universitext},
  publisher = {Springer New York},
  address   = {New York, NY},
  year      = {2010}
}

@article{Nowicki:1988:ROC,
author = {Andrzej Nowicki and Masayoshi Nagata},
title = {{Rings of constants for $k$-derivations in $k[x_1, \ldots, x_n]$}},
volume = {28},
journal = {Journal of Mathematics of Kyoto University},
number = {1},
publisher = {Duke University Press},
pages = {111 -- 118},
year = {1988},
doi = {10.1215/kjm/1250520561}
}

@article{Nowicki:1994:RFC,
title = {Rings and fields of constants for derivations in characteristic zero},
journal = {Journal of Pure and Applied Algebra},
volume = {96},
number = {1},
pages = {47-55},
year = {1994},
doi = {10.1016/0022-4049(94)90086-8},
author = {Andrzej Nowicki}
}

@article{Okuda:2004:KOD,
author = {Shun-Ichiro Okuda},
title = {{Kernels of derivations in positive characteristic}},
volume = {34},
journal = {Hiroshima Mathematical Journal},
number = {1},
publisher = {Hiroshima University, Mathematics Program},
pages = {1 -- 19},
year = {2004},
doi = {10.32917/hmj/1150998069}
}

@article{Santos:2023:PCD,
title = {Positive characteristic {D}arboux-{J}ouanolou integrability of differential forms},
journal = {Journal of Pure and Applied Algebra},
volume = {227},
number = {2},
pages = {107195:1--15},
year = {2023},
doi = {10.1016/j.jpaa.2022.107195},
author = {Edileno de Almeida Santos and Sergio Rodrigues}
}

\end{document}